\documentclass[11pt,a4paper]{amsart}
\usepackage{amsmath,amsfonts}
\usepackage{graphicx}
\usepackage{amssymb}
\usepackage{amsthm}
\usepackage{yfonts}
\usepackage{kpfonts}
\usepackage{float}

\usepackage[usenames,dvipsnames]{pstricks}
\usepackage{epsfig}
\usepackage{rotating}
\usepackage{pst-plot}
\usepackage{pst-eps}
\usepackage{pst-grad}
\usepackage{pstricks-add}
\usepackage{lmodern}
\usepackage{xcolor}
\usepackage{graphicx}
\usepackage[top=2cm,bottom=2cm,left=2cm,right=2cm,a4paper]{geometry}

\usepackage{etex}
\usepackage[all]{xy}

\def\eps{{\varepsilon}}

\newcommand{\qq}{{\mathbb Q}}

\newcommand{\cc}{{\mathbb C}}
\newcommand{\zz}{{\mathbb Z}}

\newcommand{\pp}{{\mathbb P}}
\newcommand{\projgr}{{\mathbb G}}

\newcommand{\Pic}{\hbox{Pic}}

\newcommand{\sh}{\mathscr}
\newcommand{\cal}{\mathcal} 

\newcommand{\Mg}{\mathcal{M}_g}
\newcommand{\Mgbar}{\overline{\mathcal{M}}_g}

\newcommand{\Mgn}{\mathcal{M}_{g,n}}
\newcommand{\Mgnbar}{\overline{\mathcal{M}}_{g,n}}
\newcommand{\Fg}{\mathcal{F}_g}
\newcommand{\Hgmubar}{\overline{\mathcal{H}}_g(\mu)}

\newcommand{\Hgmu}{\mathcal{H}_g(\mu)}

\theoremstyle{plain}
\newtheorem{theorem}{Theorem}[section]
\newtheorem{lemma}[theorem]{Lemma}
\newtheorem{proposition}[theorem]{Proposition}
\newtheorem{corollary}[theorem]{Corollary}

\theoremstyle{definition}
\newtheorem{remark}[theorem]{Remark}
\newtheorem{definition}[theorem]{Definition}

\usepackage[alphabetic]{amsrefs}

\title{Uniruledness of Strata of Holomorphic Differentials in Small Genus}
\author{Ignacio Barros}
\address{
Humboldt-Universit\"{a}t zu Berlin\\
Institut f\"{u}r Mathematik\\
Unter den Linden 6, 10099 - Berlin\\ 
Germany.} 
\email[]{iabarros@math.hu-berlin.de}

\begin{document}

\maketitle

\begin{abstract}
We address the question concerning the birational geometry of the strata of holomorphic and quadratic differentials. We show strata of holomorphic and quadratic differentials to be uniruled in small genus by constructing rational curves via pencils on K3 and del Pezzo surfaces respectively. Restricting to genus $3\leq g\leq6$, we construct projective bundles over a rational varieties that dominate the holomorphic strata with length at most $g-1$, hence showing in addition that these strata are unirational.
\end{abstract}


\nopagebreak



\section*{Introduction}

Let $g\geq2$ be an integer and $\mu=(m_1,\ldots,m_n)$ an integer partition of $2g-2$. The moduli space of canonical divisors of type $\mu$ is defined as the closed substack $\Hgmu\subset\Mgn$ given by 
$$\Hgmu=\left\{\left[C,x_1,\ldots,x_n\right]\in\Mgn\left|\hspace{0,1cm} \cal{O}_C\left(\sum_{i=1}^{n}m_ix_i\right)\cong\omega_C\right.\right\}.$$
A partition $\mu$ is referred to as \textit{holomorphic} if all entries are non-negative. To some extent Diaz in \cite{StevenDiazTangent}, and more precisely and in a greater generality Polischuck in \cite[Thm 1.1.a]{2003math......9217P}, proved that $\Hgmu$ is nonsingular and of pure codimension $g-1$ when $\mu$ is holomorphic. The same arguments apply if we allow the partition $\mu$ to have negative terms. In that case $\Hgmu$ is again nonsingular, but of pure codimension $g$ inside $\Mgn$. Kontsevich and Zorich \cite[Thm. 1 and 2]{2003InMat.153..631K} showed that for a holomorphic partition, $\Hgmu$ may have up to three connected or equivalently irreducible components depending on the parity of the entries and the length of the partition. For example, observe that if all entries of $\mu$ are even,
$$\eta=\cal{O}_{C}\left(\frac{1}{2}\sum m_ix_i\right)$$
is a theta characteristic and the parity of $h^0(C,\eta)$ is stable under deformations, cf. \cite{Mumf1971}. If $\mu=(2,\ldots,2)$, there is a dominant forgetful map
$$\begin{array}{rcl}
\Hgmu&\to&\sh{S}_g^-\coprod Z\\
\left[C,x_1,\ldots,x_{g-1}\right]&\mapsto&\left[C,x_1+\ldots+x_{g-1}\right]
\end{array}$$
where $\sh{S}_g^-$ is the space of odd theta characteristics and $Z\subset \sh{S}_{g}^{+}$ is the divisor of even theta characteristics with at least two sections. In this case our space breaks into at least two components.\\

A modular interpretation of the boundary $\partial\Hgmubar$ of the Zariski closure
$$\Hgmubar\subset\Mgnbar$$
has been the subject of much recent attention. Farkas and Pandharipande \cite{2015arXiv150807940F} defined a proper moduli space, called the moduli space of {\it{twisted canonical divisors}} which it contains our space of interest as an open subset. However, the question of which twisted canonical divisors in the boundary lie in $\Hgmubar$ requires more information than the dual graph of the curve and the relations in the Jacobian of each component. The boundary components are parametrized not just by the topological type of the nodal curves and strata in smaller genera, but also by the particular complex structure, manifesting as residue conditions at the nodes provided by Bainbridge, Chen, Gendron, Grushevsky and M\"{o}ller \cite{2016arXiv160408834B}.\\

The next natural question concerns the global geometry of the strata $\Hgmubar$. In this paper we establish a range for the genus and length of partition where the Kodaira dimension is negative. For holomorphic partitions, the moduli space of odd spin curves $\overline{\cal{S}}_g^{-}$ serves as a useful model highlighting the change in the birational type of these varieties from uniruled to general type and we follow the general strategy presented in \cite{2010arXiv1004.0278F}.\\

We describe our strategy. When the length of $\mu$ is greater than $g-1$, the forgetful map $\Hgmu\to\Mg$ is dominant and for small enough genus a canonically embedded smooth curve can be realized as a hyperplane section of a K3 surface $S\subset \pp^g$. Since the divisor $\sum m_ix_i$ is canonical, it is the intersection of a codimension $2$ plane with $S$. The rational curve in $\Hgmubar$ passing through a general point is induced by the pencil spanned by this plane in $\pp^g$.\\  

The genus $g=10$ case is especially delicate and we treat it by studying $1$-nodal models of geometric genus $10$ curves lying on a K3 surface. For generic choice of a pointed curve $[C,x_1,x_2]\in\mathcal{M}_{10,2}$, we prove that we can always find such models. We do this by deformation theory arguments that reduce the question of dominance of a moduli map to cohomology computations. \\  

For holomorphic partitions $\mu$ with $g\geq4$, Kontsevich and Zorich \cite[Thm. 1 and 2]{2003InMat.153..631K} classified the stratum having more than one connected (or irreducible) component:
\begin{itemize}
\item If $\mu=(2g-2)$ or $\mu=(2l,2l)$ with $l\geq1$, then $\Hgmu$ has three components; the hyperelliptic one $\cal{H}^{hyp}$ and the even and odd $\cal{H}^+$, $\cal{H}^{-}$ depending on the parity of the induced theta characteristic, i.e., $[C,x_1,\ldots,x_n]\in \cal{H}^+$ (resp. $\cal{H}^-$) if and only if $h^0\left(C,\cal{O}_C\left(\frac{1}{2}\sum m_ix_i\right)\right)\equiv 0$ (resp. $\equiv 1$) $\mod 2$. 
\item If $\mu=(2l_1,\ldots,2l_r)$ for $l_i\geq0$ and $r\geq3$, then $\Hgmu$ has two components $\cal{H}^+$ and $\cal{H}^-$, depending on the parity of the induced theta characteristic.
\item Finally, if $\mu=(2l-1,2l-1)$ for $l\geq 2$, then $\Hgmu$ has two components, the hyperelliptic one $\cal{H}^{hyp}$ and a complementary one $\cal{H}^{nonhyp}$. 
\end{itemize}
For any other partition the space is connected. For $g=3$, both $\mathcal{H}(2,2)$ and $\mathcal{H}(4)$ have two components; the hyperelliptic and the odd spin structure (or non-hyperelliptic) component. The strata are connected for other partitions.\\

Regarding the irreducible stratum we prove the following:
\begin{theorem}
\label{thm}
Let $\Hgmubar$ be an irreducible stratum with length of partition $l(\mu)$. The birational geometry of $\Hgmubar$ is summarized in the following table:
\begin{table}[H]
\centering
\begin{tabular}{|l|l|l|}
\hline
                   & Unirational & Uniruled \\ \hline
$3\leq g\leq 6$    & $l(\mu)\leq g-1$      & No restriction on $\mu$  \\ \hline
$g=7,8$            & \hspace{0.8cm}\hbox{?}        & No restriction on $\mu$ \\ \hline
$g=9$            & \hspace{0.8cm}\hbox{?}         & $l(\mu)\geq 7$         \\ \hline
$g=10$             & \hspace{0.8cm}\hbox{?}        & $11\leq l(\mu)< 18$         \\ \hline
$g=11$             & \hspace{0.8cm}\hbox{?}        & $l(\mu)\geq 10$         \\ \hline
\end{tabular}
\end{table}
\end{theorem}

For those where irreducibility does not holds we prove the following:

\begin{theorem}
\label{non-irred}
For every genus $\cal{H}^{hyp}(2g-2)$ is unirational. Even and odd strata are uniruled for any partition in the range $g\leq 8$. For partition $\mu=(2,\ldots,2)$, the odd stratum $\cal{H}^{-}$ is uniruled in genus $g=9,11$ and the $S_{g-1}$-quotient of the even stratum $\cal{H}^{+}$ is uniruled for every genus.
\end{theorem}

Of equal interest is the strata of $k$-differentials when $k\geq 2$. In Section \ref{quadratic} we treat the case $k=2$. Let $\nu=(n_1,\ldots,n_{l(\nu)})$ be a partition of $4g-4$ with length $l(\nu)$ and 
$$\cal{Q}(\nu)=\left\{\left[C,x_1,\ldots,x_{l(\nu)}\right]\in\cal{M}_{g,l(\nu)}\left|\hspace{0,1cm} \cal{O}_C\left(\sum_{i=1}^{l(\nu)}n_ix_i\right)\cong\omega_C^{\otimes 2}\right.\right\}$$
the strata of quadratic differentials. The space $\cal{Q}(\nu)$ is smooth and can have several connected components due to hyperellipticity and the parity of theta characteristics when every entry of $\nu$ is even, among others. For $g\geq 3$ and a holomorphic partition $\nu$, if at least one entry is odd, the space $\mathcal{Q}(\nu)$ is smooth, connected and of dimension $2g-3+l(\nu)$, with the only exception of $g=4$ and $\nu=(3,3,3,3)$, where the space has several connected components. See \cite[Thm. 3 and 4]{2002math.....10099L}, \cite{Veech1990}, \cite[Thm. 1.1]{johannesS} and \cite[Thm 7.1]{CM14}. In Section \ref{quadratic} we turn to del Pezzo surfaces to construct rational curves on $\cal{Q}(\nu)$, obtaining the following result:

\begin{theorem}
\label{quad.diff.}
For genus $3\leq g\leq 6$ and partition $\nu$ with at least one odd entry, different from $(3,3,3,3)$ and length $l(\nu)\geq g$, the moduli space of quadratic differentials $\cal{Q}(\nu)$ is uniruled.
\end{theorem}

\subsection*{Nodal Curves on $K3$ Surfaces.}

A \textit{K3 surface} $S$ is a smooth connected complex projective surface with $K_S\cong\cal{O}_S$ and $h^1(S,\cal{O}_S)=0$. A \textit{polarized} K3 surface of genus $g$ is a pair $(S,H)$ where $H$ is a primitive, nef line bundle on $S$, $|H|$ has no fixed components and $H^2=2g-2$ (hence big for $g\geq2$). Such pairs form a smooth moduli space denoted $\cal{F}_g$. Given such a pair $(S,H)\in\cal{F}_g$, for $g\geq2$ the linear system $|H|$ is base point free of dimension $g$ and if it does not contain hyperelliptic curves, the induced map 
$$\psi_{|H|}:S\to\pp^g$$  
is birational onto the image, see e.g. \cite[Thm. 8.3]{gralK3} for $H$ ample and \cite[Lec. 9]{ClKo88} or \cite[Thm. 7.32]{De01} for $H$ big and nef. The image of $\psi_{|H|}$ is a degree $2g-2$ normal surface on $\pp^g$, with at worst canonical singularities and general hyperplane sections are canonical curves of genus $g$. \\

Given a positive integer $\delta\leq g$, the \textit{Severi variety} of irreducible $\delta$-nodal curves in the linear system $|H|$ is denoted by $V_{\delta}(S,H)$. It is well-known that for $(S,H)\in\cal{F}_g$ general and $g\geq 2$, $V_\delta(S,H)$ is non-empty and each irreducible component is smooth of dimension $g-\delta$. We refer to \cite{Chen99}, \cite[Cor. 1.2]{Chen16}, \cite{nodalK3} and \cite{2000math......4130F} for fundamental facts on this matter.

\begin{definition}
For integers $g$ and $\delta$, with $g\geq3$ and $0\leq\delta\leq g-2$, we define the \textit{universal Severi variety} $\cal{V}_{g,\delta}$ to be the algebraic stack of pairs $(S,X)$ with $(S,H)\in\cal{F}_g$ and $X\in V_\delta(S,H)$. 
\end{definition}

The stack $\cal{V}_{g,\delta}$ is smooth and every irreducible component has dimension $19+(g-\delta)$. It was conjectured by C. Ciliberto and T. Dedieu \cite[Thm. 2.1]{CilDed2012} that $\cal{V}_{g,\delta}$ is irreducible and this was proved in the range $3\leq g\leq 11$, $g\neq 10$ and $0\leq\delta\leq g$. The natural forgetful map 
$$\pi_\delta:\cal{V}_{g,\delta}\to\cal{F}_g$$
is smooth and when restricted to any irreducible component is dominant; see \cite{2007arXiv0707.0157F}.\\

There is a well-defined moduli map 
$$c_{g,\delta}:\cal{V}_{g,\delta}\to\cal{M}_{g-\delta}$$
sending a pair $(S,X)$ to the isomorphism class of the normalization of $X$. When $\delta=0$ this is Mukai's map: dominant for $g\leq 9$ and $g=11$, not dominant for $g=10$ and generically finite over the image for $g\geq 11$ and $g\neq 12$. See \cite{2MukaiK3}, \cite{1MukaiK3}, \cite{4MukaiK3} and \cite{MoMu83}. Moreover, for $g=11$ and $g\geq 13$, Mukai's map is birational over the image, cf. \cite[Thm. 4.5]{CLM93}. \\

Beauville in \cite[\S 5]{2002math.....11313B} studied the differential of the map $c_{g,0}$ and gave a deformation theoretic proof of Mukai's results. Flamini, Knutsen, Pacienza and Sernesi proved in \cite[Thm. 1.1]{2007arXiv0707.0157F} that in the range $3\leq g\leq 11$ and $0\leq \delta\leq g-2$, if $V\subset \cal{V}_{g,\delta}$ is any irreducible component, then for $2\leq g-\delta<g\leq 11$ the restricted map 
$$c_{g,\delta}\mid_V:V\to \cal{M}_{g-\delta}$$
is dominant with general fiber of dimension $22-2(g-\delta)$. This result was generalized in \cite[Thm. 1.1]{Ke15} and \cite[Thm. 1.1]{CFGK17} for nodal curves in the linear system $|kH|$. They proved dominance in low genus and generic finiteness for $g$ high enough. \\

Let $(S,X)\in\cal{V}_{g,\delta}$ be a $\delta$-nodal curve on a K3 and $\nu:C\to X$ the normalization map with $\nu^{-1}(\hbox{Sing}(X))=\{p_1,q_1,\ldots,p_\delta,q_\delta\}$. We can consider the following moduli map
$$\begin{array}{rcl}
c:\cal{V}_{g,\delta}&\to&\cal{M}_{g-\delta,[2\delta]}\\
(S,X)&\mapsto&\left[C,p_1+q_1+\ldots+p_\delta+q_\delta\right].
\end{array}$$

The following theorem is the main ingredient to treat the genus $10$ case. Following \cite{2007arXiv0707.0157F} we were able to extend their result.

\begin{theorem}
\label{thmDT}
In the range $3\leq g\leq 11$, $1\leq \delta\leq g-2$ and $g\neq 10$, the moduli map
$$c:\cal{V}_{g,\delta}\to\cal{M}_{g-\delta,[2\delta]}$$
is dominant with general fiber dimension $22-2g$. Moreover, when $g=10$, the image of $c$ has always codimension one and the class is given by
$$c_*\left[\cal{V}_{10,\delta}\right]=\frac{1}{2\delta!}\left(7\lambda+\psi\right)\in \mathrm{Pic}_{\qq}(\cal{M}_{10-\delta,[2\delta]}).$$
\end{theorem}

Here the class $\psi$ is defined to be the push forward of the $S_n$-invariant cycle $\psi_1+\ldots+\psi_{2\delta}$ by the finite quotient map $\cal{M}_{g-\delta,2\delta}\to\cal{M}_{g-\delta,[2\delta]}$. 

Farkas and Popa \cite[Thm 1.6]{FarkasPopa} proved that the class of the closure in $\overline{\cal{M}}_{10}$ of the locus of smooth curves lying on a $K3$ surface is given by
$$\overline{\cal{K}}=7\lambda-\delta_0-5\delta_1-9\delta_2-12\delta_3-14\delta_4-B_5\delta_5\in\mathrm{Pic}_\qq\left(\overline{\mathcal{M}}_{10}\right),$$
with $B_5\geq 6$. The formula for $c_*[\cal{V}_{10,\delta}]$ is a pull back of $\overline{\cal{K}}$ by a boundary morphism. See \cite[Chapter 17, Lemma 4.35]{ACG2}.\\

The paper consists of an introduction and four sections. In the first section we give a full proof of the second column concerning uniruledness in Theorem \ref{thm}. We start in \S1.1 by treating the cases $3\leq g\leq 9$ and $g=11$, when the length of the partition is at least $g-1$. In \S1.2 we show uniruledness in the range $3\leq g \leq 8$ for partitions of any length and in \S1.3 we establish uniruledness in genus $g=9$ and partitions of length $l(\mu)\geq 7$. In \S1.4 we prove that uniruledness in $g=10$ with partitions of length $11\leq l(\mu)<18$ follows from Theorem \ref{thmDT} and in \S1.5 we give a full proof of Theorem \ref{thmDT}. In the second section we deal with quadratic differentials and show Theorem \ref{quad.diff.}. In the third section we establish the first column concerning unirationality in Theorem \ref{thm}. Finally, in the last section we prove Theorem \ref{non-irred}, concerning the non-irreducible cases.

\section*{Acknowledgements}
This is part of my ongoing PhD under the supervision of Gavril Farkas and Rahul Pandharipande. I can not thank enough their infinite patience, permanent support and vast insight and experience. Thanks also go to Andreas Knutsen for getting involved in some of the questions, to Scott Mullane for many helpful comments and suggestions and to the anonymous referee for a careful reading and many insightful suggestions that helped to substantially improve the exposition. Quentin Gendron was kind enough to communicate me a missing reference. Finally, I would like to thank my colleagues, especially Daniele Agostini. My PhD studies are generously supported by the Einstein Stiftung Berlin.


\section{Uniruled Strata}
 
As explained in the introduction, the general strategy is to construct pencils on K3 surfaces to prove uniruledness. 

\begin{lemma}
Let $(S,H)\in\cal{F}_g$ be a general polarized $K3$ of genus $g\geq 2$ and $P\subset|H|$ a general pencil whose general element is a smooth curve of genus $g$. Then the induced rational map $P\dashrightarrow \Mg$ is non trivial. 
\end{lemma}

\begin{proof}
As mentioned in the introduction, for general $(S,H)\in\Fg$ and $\delta\leq g$, the Severi variety of $\delta$-nodal curves 
$$V_\delta(S,H)=\left\{C\in|H|\mid C\hbox{ is $\delta$-nodal and irreducible}\right\}$$
is non-empty, regular and of codimension $g$ in $|H|$. Choosing the pencil $P$ to be general, it will intersect $V_1(S,H)$ non-trivially. Thus, the induced map to $\Mg$ cannot be trivial.   
\end{proof}
We need a slight refinement of the previous lemma.
\begin{lemma}
\label{nontrivial}
Let $(S,H)\in\cal{F}_g$ be a general polarized $K3$ of genus $g\geq 2$ and $P\subset|H|$ any pencil whose general element is a smooth curve of genus $g$. Then the induced rational map $P\dashrightarrow \Mg$ is non trivial.
\end{lemma}

\begin{proof}
By contradiction we assume the induced rational map to $\Mg$ is trivial. We blow up $S$ at the base locus of $P$:
\begin{displaymath}
\xymatrix{
\sh{U}\ar[r]^{\hbox{Bl}_ZS}\ar[d]_{\pi}&S\ar@{-->}[dl]\\
P.&
}
\end{displaymath}
Since the automorphism group of $\sh{U}$ is finite, up to an \'{etale} base change $B\to P$, 
$$\tilde{\sh{U}}:=\sh{U}\times_{P}B$$
is birational to $C\times B$, cf. \cite[Thm. 2]{MaMu64}. Notice that $\sh{U}$ is simply connected, therefore, $\tilde{\sh{U}}\cong \sh{U}$ and $B\cong\pp^1$. It follows that $S$ is birational to $C\times \pp^1$, but $S$ is not ruled.
\end{proof}

Moreover, for $(S,H)\in\cal{F}_g$ general, every curve on a general pencil $P\subset |H|$ is irreducible and at worst nodal. The base locus of $P$ consist of $2g-2$ points. If we resolve the map $S\dashrightarrow P$ by blowing up, we obtain a family $\sh{U}$ over $P$ as in the proof of the lemma above. The general fiber $F$ is a smooth genus $g$ curve. From the relation
$$\chi(\sh{U},\zz)=\chi(P,\zz)\cdot\chi(F,\zz)+\hbox{number of singular fibers}$$
one deduces that on $\Mgbar$, $P\cdot\delta_{irr}=6g+18$, where $\delta_{irr}$ is the divisor on $\Mgbar$ whose general element correspond to a one-nodal irreducible curve. In other words, the hypersurface parametrizing singular curves in the linear system $D_{S,H}\subset|H|$ has degree $6g+18$.

\subsection{General Case, $g\leq 11$ and $g\neq 10$}
\label{main_argument}

Recall that a variety $X$ is uniruled if for a general point $p\in X$ there is a rational curve passing through it. In other words if there exists a variety $Y$ and a dominant rational map $Y\times \pp^1\dashrightarrow X$.
Let $\mu=(m_1,\ldots,m_n)$ be an holomorphic partition of $2g-2$ with length $n$ and 
$$\left[C,x_1,\ldots,x_n\right]\in\Hgmu$$
a point on the stratum. Assume $3\leq g\leq 9$ or $g=11$. The forgetful map $\pi:\Hgmu\to\Mg$ is dominant when the length of the partition is greater than or equal to $g-1$; see \cite{2015arXiv150303338G}. The curve $C$ is general and, therefore, can be embedded as a hyperplane section on a genus $g$ polarized K3 surface $(S,H)\in\cal{F}_g$. See \cite{2MukaiK3} and \cite{1MukaiK3}. Here $\cal{F}_g$ is the moduli space of polarized K3 surfaces $(S,H)$, where $S$ is a K3 surface and $H\in \Pic(S)$ is a (primitive) polarization of degree $H^2=2g-2$. We construct a rational curve 
$$\pp^1\to \Hgmubar$$
passing through $[C,x_1,\ldots,x_n]$. Our curve is embedded in $S$ as hyperplane section
$$C\cong S\cap H\hookrightarrow S\subset \pp H^0(S,H)^{\vee}\cong \pp^g.$$
The divisor $m_1x_1+\ldots+m_nx_n\in \hbox{Div}(C)$ is canonical so can be realized as a hyperplane section of $H\cong\pp^{g-1}$, i.e., a point $\Lambda_\mu\in \projgr(g-2,\pp^g)$ such that 
$$\Lambda_\mu\cdot S=m_1x_1+\ldots+m_nx_n.$$ 
Let 
$$P\cong\{H'\in \left(\pp^g\right)^{\vee}\mid \Lambda_\mu\subset H'\}$$
be the pencil of hyperplanes containing $\Lambda_\mu$ in $\pp^g$. Since $C\in P$ is smooth, for a general hyperplane $H'\in X$, the curve $C'=H'\cap S\hookrightarrow H'\cong \pp^{g-1}$ is smooth and canonically embedded. Moreover, the hyperplane $\Lambda_\mu\subset \pp^{g-1}$ is a canonical divisor of the form 
$$\Lambda_\mu\cdot S=\Lambda_\mu\cdot C=m_1x_1+\ldots+m_nx_n.$$ 
This construction gives us a map defined on an open subset of $P\cong \pp^1$

$$\begin{array}{rcl}
\gamma: \pp^1&\dashrightarrow&\Hgmubar\\
H'&\mapsto&\left[H'\cap S,x_1,\ldots,x_n\right].
\end{array}$$
The map can be extended and we already proved in Lemma \ref{nontrivial} that it cannot be trivial. This gives us Theorem \ref{thm} for $g\leq 9$ and $g=11$ when the length of $\mu$ is at least $g-1$.

\subsection{Special Cases for Genus $g\leq 8$.} 

Let $C$ be a smooth curve of genus $g\geq 2$. We recall a result of M. Ide: 

\begin{theorem}[\cite{Ide}]
Every smooth curve $C$ of genus $2\leq g\leq 8$ can be embedded in a smooth $K3$ surface $S$ with $C\subset S$ big and nef. 
\end{theorem}

In \cite{Ide} it is proved that all smooth curves of genera $2\leq g\leq 8$ can be embedded as ample classes in the smooth locus of a K3 with at worst rational double points. It is a standard fact that every complex projective surface $S$ with at worst ADE singularities admits a crepant resolution (see \cite{ypgR}). In our case $S$ is a K3 surface with at worst rational double points, so there is a unique resolution 
$$\pi:\tilde{S}\to S.$$
The resolution is crepant meaning $\tilde{S}$ is again a K3 and if $\tilde{C}$ is the proper transform of $C$, as divisor $\tilde{C}$ might cease to be ample (it can have trivial intersection with $(-2)$-curves) but it is still big and nef.

Let $[C,x_1,\ldots,x_n]\in \Hgmu$ be a general point on a connected (irreducible) stratum with $3\leq g\leq 8$ and let $S$ be a big and nef K3 extension of $C$. The map $\phi_{C}:S\dashrightarrow \hat{S}\subset\pp^g$ restricted to $C$ is the canonical map
$$\phi_{K_C}:C\to\pp^{g-1}.$$
The point $[C,x_1,\ldots,x_n]$ is general and we are under the assumption that $\Hgmu$ is connected. Therefore, $C$ is not hyperelliptic and $\phi_{K_C}$ is an embedding. We can repeat the same construction as for the general case. Since the general hyperplane of $\pp^g$ in the pencil of hyperplanes through $\Lambda_\mu$ as before intersects $\hat{S}$ in a smooth curve, the pull back is smooth (it does not contains $(-2)$-curves).

This gives us uniruledness for every irreducible stratum $\Hgmubar$ in the range $3\leq g\leq 8$.

\subsection{Genus $g=9$.} We have already proven in \S1.1 that $\mathcal{H}_9(\mu)$ is uniruled for $l(\mu)\geq8$. We can improve the lower bound by one. For small length partitions the forgetful map 
$$\pi:\Hgmu\to\Mg$$
is no longer dominant and in order to carry out the argument above one has to prove that the image of Mukai's map 
$$\begin{array}{rcl}
\mathcal{V}_{g,0}&\to& \Mg\\
(S,H,C)&\mapsto&\left[C\right]
\end{array}$$ 
intersects $\pi(\Hgmu)$ on a non-empty open in $\pi(\Hgmu)$. A smooth complex curve of genus $9$ can be realized as an hyperplane section of a K3 if its not pentagonal (has no $g^1_5$), cf. \cite[Thm. A]{3MukaiK3}. In particular the image of $\mathcal{V}_{9,0}\to \cal{M}_9$ contains the complement of the Brill-Noether divisor $D^1_5$ consisting of pentagonal curves. Recall that $\hbox{Pic}_{\qq}\left(\Mgbar\right)$ is generated by $\lambda$ and the boundary divisors $\delta_{irr}, \delta_1,\ldots,\delta_{\lfloor g/2\rfloor}$ and the slope of a divisor
$D=a\lambda-b_{irr}\delta_{irr}-\sum_{1\leq i\leq \lfloor g/2\rfloor}b_i\delta_i$
not containing any boundary components is defined to be 
$$s(D):=\frac{a}{\min\{b_{irr},b_1,\ldots,b_{\lfloor g/2\rfloor}\}}.$$
One can check that if $D$ and $D'$ are two divisors on $\Mg$, with $\overline{D}$ and $\overline{D}'$ their closures inside $\Mgbar$, then $\overline{D}\leq \overline{D}'$ implies that $s(\overline{D})\leq s(\overline{D}')$. S. Mullane \cite[\S 5]{2015arXiv150903648M} computed the class of the closure of the image $\Hgmu\to\Mg$ when $l(\mu)=g-2$. In genus $9$ the slope is strictly bigger than eight, in particular, bigger than the slope of the pentagonal locus (cf. \cite{HaMo90})
$$s(\overline{D}_5^1)=6+\frac{12}{10}.$$
In any case, when $\overline{\mathcal{H}}_9(\mu)$ is irreducible and $l(\mu)=g-2$, the classes
$$\pi_*\left[\overline{\mathcal{H}}_9(\mu)\right]\hspace{0.3cm}\hbox{and}\hspace{0.3cm}\overline{D}_5^1$$ 
are not proportional. Moreover, since the slope of $\pi_*\left[\overline{\mathcal{H}}_9(\mu)\right]$ is strictly bigger than $s(D_5^1)$, the image of $\overline{\mathcal{H}}_9(\mu)$ cannot be contained in the pentagonal locus. We can conclude that for a general point $[C,x_1,\ldots,x_{7}]\in \mathcal{H}_9(\mu)$, with partition length $l(\mu)=7$, the curve $C$ can be embedded in a K3 surfaces and the argument above can still be carried out. This finishes the proof of Theorem \ref{thm} for $g=9$.

\begin{remark}
By specialization, it would be enough to show that there is a Brill-Noether general curve in the non-hyperelliptic component of $\cal{H}_g(2g-2)$, but to construct a Brill-Noether general curve $C$ admitting a subcanonical point is not an easy task.
\end{remark}

\subsection{Genus $10$} 
The genus $10$ case is much more delicate and it is done by studying irreducible nodal curves of arithmetic genus $11$. We define the set $\cal{U}_{10}\subset \cal{H}_{10}(\mu)$ by the condition 
$$[C,x_1,\ldots,x_n]\in \cal{U}_{10}$$ 
if and only if there exists a polarized K3 surface $(S,H)\in\cal{F}_{11}$ and a non trivial map $f:C\to S$, such that $f_*[C]\in \left|H\right|$ and $f$ is the normalization map of the irreducible nodal curve $f(C)$ having a single node at $f(x_1)=f(x_2)$.

\begin{proposition}
Through every point of $\cal{U}_{10}\subset\overline{\cal{H}}_{10}(\mu)$ there passes a rational curve. 
\end{proposition}

\begin{proof}
Let $[C,x_1,\ldots,x_n]$ be a point on $\cal{U}_{10}$ and 
$$\epsilon:\tilde{S}\to S$$
the blow-up of $S$ at the node $f(x_1)=f(x_2)$. The curve $C$ can be embedded in $\tilde{S}$. Moreover, $C\in \left|\epsilon^*H-2E\right|$, where $E$ is the exceptional divisor of $\epsilon$. By adjunction
$$\cal{O}_{C}\left(C\right)\cong K_{C}\left(-x_1-x_2\right)\cong\cal{O}_{C}\left((m_1-1)x_2+(m_2-1)x_2+\sum_{i\geq 3}m_ix_i\right).$$
Let us assume that $l(\mu)\geq 2$. The divisor 
$$D=(m_1-1)x_2+(m_2-1)x_2+\sum_{i\geq 3}m_ix_i$$
is effective on $C$. The exact sequence 
$$0\to\cal{O}_{\tilde{S}}\to\cal{I}_{D\big/\tilde{S}}\left(C\right)\to\cal{O}_{C}\to 0,$$
where the middle term is the ideal sheaf of the closed subscheme $D\subset \tilde{S}$ twisted by $C$, proves that
$$\pp H^0\left(\tilde{S},\cal{I}_{D\big/\tilde{S}}\left(C\right)\right)\cong \pp^1.$$
There is a rational map 
$$\pp^1\dashrightarrow \cal{U},$$
sending the generic element $C'\in |\eps^*H-2E|$ passing trough $D$ to 
$$C'\in \pp^1\mapsto[C',x_1,x_2,\ldots,x_n].$$
The same argument as in Lemma \ref{nontrivial} applies to prove that the induced map to $\Mg$ is not trivial.

We can assume $m_1\geq m_2\geq\ldots\geq m_n$. Notice that the argument fails when the set $\{x_1,x_2\}$ is not contained in the support of $D$. If $m_1>m_2=1$ we can still keep track of the points since $x_1\in \hbox{Supp}(D)$ and for $C''$ general, $C''\cap E=x_1+q$. We impose $x_2=q$ and the argument still holds true. 
\end{proof}

\begin{remark}
When $g=10$, the only partition of maximal length $l(\mu)=18$ is 
$$\mu=(1,\ldots,1).$$ 
Our map is well defined in the quotient
$$
\begin{array}{rcl}
\pp^1&\dashrightarrow&\cal{U}_{10}\big/\left(\zz\big/2\zz\right)\\
C''&\mapsto&\left[C'',y_1+y_2,x_3,\ldots,x_n\right]
\end{array}
$$
where $y_1+y_2=C''\cap E$ and we have uniruledness for the quotient 
$$\cal{U}_{10}\to\cal{U}_{10}\big/\left(\zz\big/2\zz\right).$$
\end{remark}

Consider the cartesian diagram
\begin{displaymath}
\xymatrix{
\cal{V}_{11,1}\times\cal{H}_{10}(\mu)\ar[r]^{p_2}\ar[d]_{p_1}&\cal{H}_{10}(\mu)\ar[d]^{\pi}\\
\cal{V}_{11,1}\ar[r]^{c}&\cal{M}_{10,[2]}.
}
\end{displaymath}
Notice that the image of $p_2$ is exactly $\cal{U}_{10}$, and in the range $l(\mu)\geq g+1$ the map $\pi$ is dominant. It remains to prove that $c$ is dominant to conclude that $\cal{U}_{10}$ dominates the strata and uniruledness follows. 

\begin{corollary}[of Theorem \ref{thmDT}]
For every holomorphic partition $\mu$ of length $18>l(\mu)\geq 11$ the inclusion defined above $\cal{U}_{10}\hookrightarrow\cal{H}_{10}(\mu)$ is dominant.
\end{corollary}

\begin{corollary}[of Theorem \ref{thmDT}]
For every partition $\mu$ of length $11\leq l(\mu)<18$, if $\overline{\cal{H}}_{10}(\mu)$ is irreducible, then is uniruled.\end{corollary}

This would conclude the missing case $g=10$ in Theorem \ref{thm}. In the coming section we give a proof of Theorem \ref{thmDT}.


\subsection{Deformation theory of nodal curves on K3 surfaces.}
We recall a few facts about deformation theory of nodal curves on K3 surfaces. For details we refer to \cite{sernesi2006deformations} and \cite{2007arXiv0707.0157F}.

Locally trivial deformations of the pair $(S,X)$ are governed by the sheaf $T_S\langle X\rangle$ defined to be the preimage of $T_X\subset T_S\mid_X$ under the restriction $T_S\to T_S\mid_X$. It sits in two standard exact sequences 
$$0\to T_S(-X)\to T_S\langle X\rangle\to T_X\to 0$$ 
and
\begin{equation}\label{seq1}0\to T_S\langle X\rangle\to T_S\to N'_{X/S}\to 0,\end{equation}
where $N'_{X/S}$ is the \textit{equisingular normal sheaf} of $X$ in $S$. This sheaf governs the deformation theory when $S$ is fixed. Moreover,
$$T_{X}V_{\delta}(S,H)=H^0(X,N'_{X/S}).$$
The first order locally trivial deformations of the pair $(S,X)$ are parametrized by $H^1(S,T_S\langle X\rangle)$. Obstructions are parametrized by $H^2$ and local automorphisms by $H^0$. The theory is unobstructed and the coarse moduli $\cal{V}_{g,\delta}$ is smooth as a stack \cite[Prop 4.8]{2007arXiv0707.0157F}. For any $(S,X)$
\begin{equation}
\label{van0}
h^2(T_S\langle X\rangle)=h^0(T_S\langle X\rangle)=0
\end{equation}
and 
$$T_{(S,X)}\cal{V}_{g,\delta}=H^1(S,T_S\langle X\rangle).$$

Much more can be said here. Given such a pair $(S,X)$ there is a unique embedded resolution of $X$ given by the following diagram
\begin{displaymath}
\xymatrix{
C\cap E\ar[d]\ar@{^(->}[r]&C\ar[d]^{f}\ar@{^(->}[r]&\tilde{S}\ar[d]^{\eps}\\
Sing(X)\ar@{^(->}[r]&X\ar@{^(->}[r]&S,
}
\end{displaymath}
where $\tilde{S}$ is the blow-up of $S$ along the nodes, $E=E_1+\ldots+E_\delta$ is the exceptional divisor, $C$ is a smooth genus $g-\delta$ curve and $f:C\to X\subset S$ is the normalization map. Let us take a look at the tangent exact sequence for the normalization map  
$$0\to T_C\to f^*T_S\to N_f\to 0.$$
Here $N_f$ is the normal sheaf of the map $f:C\to S$. With this notation \cite[Lemma 4.16]{2007arXiv0707.0157F} 
\begin{equation}\label{eq0}f_*(N_f)=N'_{X/S}\hbox{ and }H^i(N'_{X/S})\cong H^i(N_f)\hbox{ for } i=0,1.\end{equation}
This is not surprising, since the group $H^0(N_f)$ can be identified with the tangent space at $[f]$ of the space of maps $f:C\to S$ from genus $g-\delta$ smooth curves to a fixed target with $f_*C=H$. There is a one to one correspondence between $\delta$-nodal curves on $S$ in the linear system $|H|$ and maps $f$. The correspondence is given by normalization 
$$\begin{array}{rcl}
V_\delta(S,H)&\to&M_{g-\delta}(S,H)\\
X\subset S&\mapsto& f:C\to S.
\end{array}$$
To recover the differential of the map $c_{g,\delta}$ we go to $\tilde{S}$. Consider the following diagram as in \cite{2007arXiv0707.0157F}
\begin{equation}\label{diag}
\begin{gathered}
\xymatrix{
&0\ar[d]&0\ar[d]&&\\
&\eps^*T_S(-C)\ar[d]\ar[r]^{\cong}&\eps^*T_S(-C)\ar[d]&&\\
0\ar[r]&\sh{F}_C\ar[d]^{\tau}\ar[r]&\eps^*T_S\ar[d]\ar[r]^{\lambda}&N_f\ar@{=}[d]\ar[r]&0\\
0\ar[r]&T_C\ar[d]\ar[r]&f^*T_S\ar[d]\ar[r]&N_f\ar[r]&0,\\
&0&0&&
}
\end{gathered}
\end{equation}
where $\sh{F}_C$ is the kernel of the composition $\lambda:\eps^*T_S\to f^*T_S\to N_f$. It turns out \cite[Prop. 4.22]{2007arXiv0707.0157F} that 
\begin{equation}
\label{van2}
\eps_*\sh{F}_C\cong T_S\langle X\rangle\hspace{0.3cm}\hbox{and}\hspace{0.3cm}H^i(S,T_S\langle X\rangle)\cong H^i(\tilde{S},\sh{F}_C)\hbox{ for } i=0,1,2.
\end{equation}
Moreover, $H^1(\tau)$ is the differential of the map $c_{g,\delta}$. The four authors proved that in the desired range, for a general $(S,H,C)$ in any irreducible component of $\mathcal{V}_{g,\delta}$, the map $H^1(\tau)$ is surjective, see \cite[Thm. 5.1]{2007arXiv0707.0157F}.

\begin{proposition}
\label{prop1}
With the same notation as above $\sh{F}_C\mid_E\cong \mathcal{O}_E(-1)^{\oplus 2}$.
\end{proposition} 

\begin{proof}
Let $j:E_i\hookrightarrow \tilde{S}$ be the closed embedding of one of the components of the exceptional divisor. Notice that $\eps^*T_S\mid_{E_i}\cong \cal{O}_{E_i}^{\oplus 2}$ and $N_f\cong \omega_C$. From the second row in (\ref{diag}) we have
$$
0\to L^1j^*\omega_C\to \sh{F}_C\mid_{E_i}\to \cal{O}_{E_i}^{\oplus 2}\to \cal{O}_{E_i\cap C}\to 0
$$
and $C\cap E_i=p_i+q_i$. Then, 
$$\deg(\sh{F}_C\mid_{E_i})=\deg L^1j^*\omega_C-2.$$
On the other hand by adjunction $\omega_C\cong\cal{O}_C(E+C)$ and pulling back by $j$ the exact sequence
$$0\to\cal{O}_{\tilde{S}}(E)\to\cal{O}_{\tilde{S}}(E+C)\to\omega_C\to0$$
we get 
$$0\to L^1j^*\omega_C\to\cal{O}_{E_i}(E_i)\to\cal{O}_{E_i}(E_i+C)\to\cal{O}_{p_i+q_i}\to0.$$
Counting degrees, $\deg\cal{O}_{E_i}(E_i)=-1$ and $\deg\cal{O}_{E_i}(E_i+C)=1$. Thus, $\deg L^1j^*\omega_C=0$ and $\deg\sh{F}_C\mid_{E_i}=-2$. The sheaf $\sh{F}_C\mid_{E_i}$ is free on $E_i$ of rank two and degree $-2$. By Riemann-Roch, 
$$h^0\left(E_i,\sh{F}_C\mid_{E_i}\right)=h^1\left(E_i,\sh{F}_C\mid_{E_i}\right).$$
It is enough to prove $h^0=h^1=0$. By pushing forward the second row in (\ref{diag}), we get the exact sequence (\ref{seq1}) and, by (\ref{eq0}) together with (\ref{van2}), we obtain
$$R^1\eps_*\sh{F}_C=0.$$
The exceptional locus of $\eps$ is one dimensional, therefore, $R^2\eps_*\sh{F}_C(-E_i)=0$. Thus, from 
$$0\to \sh{F}_C(-E_i)\to\sh{F}_C\to\sh{F}_C\mid_{E_i}\to 0$$
we get $R^1\eps_*\sh{F}_C\mid_E=0$. In particular, $h^1(E_i,\sh{F}_C\mid_{E_i})=h^0(E_i,\sh{F}_C\mid_{E_i})=0$.
\end{proof}
We have the following corollary:
\begin{corollary}
\label{coro1.10}
The inclusion induces isomorphisms
$$H^i(\tilde{S},\sh{F}_C(-E))\cong H^i(\tilde{S},\sh{F}_C)\hbox{ for }i=0,1,2.$$
In particular $H^1(\tilde{S},\sh{F}_C(-E))$ parametrizes locally trivial first order deformations of the pair $(S,X)$.
\end{corollary}
\begin{proof}
If we pass to cohomology in the short exact sequence 
$$0\to\sh{F}_C(-E)\to\sh{F}_C\to\sh{F}_C\mid_E\to0,$$
since $H^0(\sh{F}_C\mid_E)\cong H^1(\sh{F}_C\mid_E)=0$, we obtain our result.
\end{proof}

Consider the exact sequence obtained by tensoring the first column $(\ref{diag})$ with $\cal{O}_{\tilde{S}}(-E)$,
$$0\to\eps^*T_S(-C-E)\to\sh{F}_C(-E)\overset{\sigma}{\to} T_C\left(-\sum_{i=1}^\delta p_i+q_i\right)\to 0.$$

\begin{proposition}
\label{main-prop}
The map $$H^1(\sigma):H^1(\tilde{S},\sh{F}_C(-E))\to H^1(C,T_C(-p_1-q_1-\ldots-p_\delta-q_\delta))$$ is the differential of $c:\cal{V}_{g,\delta}\to \cal{M}_{g-\delta,[2\delta]}$.
\end{proposition}

\begin{proof}
The restriction of $\eps:\tilde{S}\to S$ to $C$ is finite so $R^{i}\eps_*$ vanishes for $i>0$. On the other hand one can check that $\eps_*T_C(-\sum p_i+q_i)$ is the tangent sheaf of the nodal curve $X$, see the proof of \cite[Lemma 27.6]{Ha10}. Again, if we apply $\eps_*$ to the second row in diagram (\ref{diag}), from (\ref{eq0}) and (\ref{van2}), we get (see also \cite[Prop. 4.22]{2007arXiv0707.0157F})
$$R^i\eps_*\sh{F}_C\cong R^i\eps_*\eps^*T_S=0\hbox{ for }i>0.$$
From the exact sequence
$$0\to\sh{F}_C(-E)\to\sh{F}_C\to\cal{O}_E(-1)^{\oplus 2}\to 0,$$
we can conclude that $R^i\eps_*\sh{F}_C(-E)=0$ for $i>0$. By Proposition \ref{prop1} and (\ref{van2}), we have isomorphisms
$$\eps_*\sh{F}_C(-E)\overset{\sim}{\to}\eps_*\sh{F}_C\cong T_S\langle X\rangle.$$
Therefore, the natural isomorphisms coming from the Leray Spectral Sequence sit in the diagram
\begin{displaymath}
\xymatrix{
H^1(\tilde{S},\sh{F}_C(-E))\ar[d]^{\cong}\ar[r]^{H^1(\sigma)}&H^1(C,T_C(-\sum p_i+q_i))\ar[d]^{\cong}\\
H^1(S,\eps_*\sh{F}_C(-E))\ar[r]^{H^1(\eps_*\sigma)}&H^1(T_X),
}
\end{displaymath}  
where the map at the bottom factors through a natural isomorphism $H^1(\eps_*\sh{F}_C(-E))\cong H^1(T_S\langle X\rangle)$ and $H^1$ of the restriction map $T_S\langle X\rangle\to T_X$ sending locally trivial first order deformations of the pair $(S,X)$ to nodal deformations of $X$. 
\end{proof}

Recall that, as stated in \cite{2007arXiv0707.0157F}, if $V_{g,\delta}\to\cal{V}_{g,\delta}$ is an \'etale atlas and $\sh{X}\hookrightarrow\sh{S}$ is the universal family of pairs $(X,S)$ induced by it, there is a universal embedded resolution
\begin{displaymath}
\xymatrix{
\sh{C}\ar@{^(->}[r]\ar[d]^{\rho}&\tilde{\sh{S}}\ar[d]^{\eps}\\
\sh{X}\ar@{^(->}[r]\ar[d]&\sh{S}\ar[dl]\\
V_{g,\delta}&
}
\end{displaymath}
where $\eps$ is the blow-up of $\sh{S}$ along the nodal locus $\cal{N}_\delta\subset \sh{X}\hookrightarrow\sh{S}$ and $\rho$ is fiber-wise over $V_{g,\delta}$, the normalization. Every locally trivial first order deformations of a pair $(X,S)$ induce one of the pair $(C,\tilde{S})$ parametrized by $H^1(\tilde{S},\sh{F}_C)$.

\begin{proof}[Proof of Theorem \ref{thmDT}]
By Proposition \ref{main-prop}, Corollary \ref{coro1.10}, the isomorphisms in (\ref{van2}) and the vanishing (\ref{van0}); 
$$\hbox{coker dc}\cong H^2(\tilde{S},\eps^*T_S(-C-E)).$$
Since $\eps^*X=C+2E$, by Serre duality and projection formula 
$$H^2(\tilde{S},\eps^*T_S(-C-E))\cong H^0(\tilde{S},\eps^*\Omega_{S}^1(X))\cong H^0(S,\Omega_S^1(X)).$$
By \cite[\S5.2]{2002math.....11313B}, the latter is $0$ when $g\leq 9$ or $g=11$ and isomorphic to $\cc$ when $g=10$. 

To compute the class of this divisor in $\hbox{Pic}_\qq(\cal{M}_{10-\delta,[2\delta]})$ we just need to pull back the class $\overline{\cal{K}}$ of the closure of the K3 locus in $\cal{M}_{10}$ by the boundary map $\xi:\overline{\cal{M}}_{10-\delta,2\delta}\to\overline{\cal{M}}_{10}$ and then push it down by the $S_{2\delta}$-quotient $\pi:\overline{\cal{M}}_{10-\delta,2\delta}\to\overline{\cal{M}}_{10-\delta,[2\delta]}$,
$$c_*[\cal{V}_{10,\delta}]=\frac{1}{n!}\pi_*\left(\xi^*\overline{\cal{K}}\right).$$
If we restrict it to $\cal{M}_{10-\delta,[2\delta]}$ we have our result.
\end{proof}

\begin{remark}
Take for example $g=10$ and $\delta=1$. Even though the normalization map to $\cal{M}_9$ is surjective, the map to $\cal{M}_{9,[2]}$ is divisorial, i.e., if $C$ is a general genus $9$ curve there is a codimension one cycle in the symmetric product without the diagonal $\Gamma\subset C^{[2]}\setminus\triangle$ consisting of points $p+q$ such that, after identifying them, the nodal curve $C/_{p\sim q}$ lies on a $K3$. Let $C^{[2]}\setminus\triangle$ be a general fiber of $\pi:\overline{\cal{M}}_{9,[2]}\to\overline{\cal{M}}_9$, since the complex structure of the curve along the fiber is constant, the Hodge bundle restricts to the trivial bundle and the class of $\Gamma$ in $\hbox{Pic}(C^{[2]}\setminus\triangle)$ is given by 
$$\frac{7\lambda+\psi_1+\psi_2}{2}\cdot \pi^*(\hbox{pt})=K_C+C.$$
The same argument works for $\delta\geq 1$.
\end{remark}


\section{Quadratic Differentials}
\label{quadratic}
Let $S$ be the blow-up of $\pp^2$ along $0\leq r \leq 8$ many point in general position. The surface $S$ is a del Pezzo surface and the class $-2K_S$ is ample. There is a moduli of such surfaces $\sh{P}_r$ realized as the quotient of an open $\cal{U}$ of $(\pp^2)^{r}$ by the group $PGL(3)$. The moduli space has dimension $\min\{2r-8,0\}$ and over it sits a natural space
$$\sh{B}_r=\left\{(S,C)\mid S\in\sh{P}_r\hbox{ and }C\in|-2K_S| \hbox{ smooth and irreducible} \right\}$$
with fibers over each del Pezzo surface $S\in \sh{P}_r$ open subsets of the projective space $|-2K_S|$. By Riemann-Roch and Kodaira Vanishing, since $\chi(\cal{O}_S)=1$;
$$\dim H^0\left(S,\cal{O}_S\left(-2K_S\right)\right)=\chi(\cal{O}_S)+3K_S^2=28-3r.$$
The fiber dimension of the map $\sh{B}_r\to\sh{P}_r$ is $27-3r$ and the dimension of $\sh{B}_r$ is $19-r$. On the other hand, if $C\in |-2K_S|$ is a smooth irreducible curve on $S$, the genus of $C$ satisfies
$$2g-2=C^2+K_S\cdot C=2K_S=18-2r$$
and there is a natural map
$$\begin{array}{rcl}
\psi_r:\sh{B}_r&\to&\cal{M}_{10-r}\\
(S,C)&\mapsto&[C].
\end{array}$$
\begin{proposition}
When $4\leq r\leq 7$, the map $\psi_r$ is dominant.
\end{proposition}

\begin{proof}
Let $[C]\in\Mg$ be a general smooth curve with $3\leq g\leq 6$. The Brill-Noether number $\rho(g,2,6)\geq0$ and the general curve $[C]\in\Mg$ has a plane nodal model $\Gamma\subset \pp^2$ of degree $6$ with $10-g$ nodes. Take $r=10-g$, $S=\hbox{Bl}_r\pp^2$ the blow-up of $\pp^2$ along the nodes, $E_1,\ldots,E_r$ the exceptional divisors and $L$ the proper transform of the line. Then the proper transform of $\Gamma$ is smooth and lies in the linear system $-2K_S=6L-2E_1-\ldots-2E_r$. 
\end{proof}

Recall that for $l(\nu)=n\geq g$, the strata of quadratic differentials $\mathcal{Q}(\nu)$ is irreducible of codimension $g$ inside $\Mgn$ if the partition $\nu$ has at least one odd entry, with the only exception of $(3,3,3,3)$ in genus four, see \S1.

\begin{lemma}
Let $\nu=(n_1,\ldots,n_{m+{g}})$ be a partition of $4g-4$ of length $m+g$ with at least one odd entry and different from $(3,3,3,3)$. The forgetful map 
$$\begin{array}{rcl}
\cal{Q}(\nu)&\to&\cal{M}_{g,m}\\
\left[C,x_1,\ldots,x_{m+g}\right]&\mapsto&\left[C,x_1,\ldots,x_m\right].
\end{array}$$
is dominant.
\end{lemma}  
\begin{proof}
The diagram
\begin{displaymath}
\label{cartesiandiagram}
\xymatrix{
\cal{Q}(\nu)\ar[r]^{i}\ar[d]^\pi&\cal{M}_{g,m+g}\ar[d]^{\sigma_\nu}\\
\mathcal{M}_{g,m}\ar[r]^{c}&\sh{J}_m^{2\cdot(2g-2)}}
\end{displaymath}
is cartesian, where $\sh{J}_m^{4g-4}$ is the universal jacobian of degree $4g-4$ over $\mathcal{M}_{g,m}$, the map $c$ is the $2$-canonical section and $\sigma_\nu$ is the global Abel-Jacobi map given by 
$$\sigma_\nu:\left[C,x_1,\ldots,x_{m+g}\right]\mapsto\left[C,x_1,\ldots,x_m,\cal{O}_C\left(\sum_{i=1}^{m+g}n_ix_i\right)\right].$$
For a smooth curve with marked points $[C,x_1,\ldots,x_m]\in\cal{M}_{g,m}$ we fix 
$$L_{[C,x]}:=\cal{O}_C\left(\sum_{i=1}^m n_ix_i\right).$$
For dimension reasons, the locus of curves $[C,x]\in\cal{M}_{g,m}$ such that $\omega^{\otimes 2}-L_{[C,x]}$ is supported on less than $g$ points, is of codimension at least one. Then if $[C,x]$ is general in $\cal{M}_{g,m}$, the Abel-Jacobi map $\sigma_\nu$ restricted to the fiber of $\sh{J}_m^{4g-4}\to\cal{M}_{g,m}$ over $[C,x]$ if given by 
$$\begin{array}{rcl}
C^{\times g}\setminus\triangle&\to&\Pic^{4g-4}(C)\\
(x_{m+1},\ldots,x_{m+g})&\mapsto& L_{[C,x]}+\cal{O}_C\left(n_{m+1}x_{m+1}+\ldots+n_{m+g}x_{m+g}\right)
\end{array}$$ 
and contains $\omega^{\otimes 2}$. Therefore, the Abel-Jacobi map $\sigma_\nu$ dominates the image of the $2$-canonical section $c$.
\end{proof}

\begin{proof}[Proof of Theorem \ref{quad.diff.}]
Let $g\leq 6$ and $\nu$ a partition of $4g-4$ with at least one odd entry, length $l(\nu)\geq g$ and different from $(3,3,3,3)$. Let $[C,x_1,\ldots,x_n]\in \cal{Q}(\nu)$ be a general point on the strata, then $C$ is general in moduli so we can assume that it lies on the image of $\psi_r$ with $r=10-g$. As in Section \ref{main_argument}, the linear system $|-2K_S|$ embeds $S$ in $\pp^{3g-3}$ and realizes $C$ as an hyperplane section 
$$C=S\cap H\subset \pp^{3g-3}.$$
The restriction of $-2K_S$ to $C$ is $2K_C$. Thus, the map $S\hookrightarrow \pp^{3g-3}$ restricted to $C$ is the $2$-canonical embedding and since $\sum n_ix_i\in \hbox{Div}(C)$ is a quadratic differential, there exists a codimension $2$ plane $\Lambda_\mu\subset \pp^{3g-3}$, with 
$$\Lambda_\mu\cdot S=\sum n_ix_i.$$
Again, let $P$ be the pencil of hyperplanes $H\in (\pp^{3g-3})^{\vee}$ containing $\Lambda_\mu$. The points $x_1,\ldots,x_n$ lie on the base locus of this pencil and  there is a rational map 
$$\begin{array}{rcl}
P&\dashrightarrow &\cal{Q}(\nu)\\
H'&\mapsto&\left[S\cap H',x_1,\ldots,x_n\right].
\end{array}$$
It is left to prove that the map is non-trivial. Let $\pi:\sh{U}\to \pp^1$ be the family of curves induced by the pencil $P$, $Z$ the base locus of $P$, and $\eps:\sh{U}\to S\to\pp^2$ the composition of the blow-up of $S$ at $Z$ with the blow-up map $\eps:S\to\pp^2$. If every fiber of $\pi$ is smooth, then the topological Euler characteristic of $\sh{U}$ is 
$$\chi(\sh{U},\zz)=\chi(\pp^1,\zz)\cdot\chi(\pi^{-1}(\hbox{point}),\zz)=2\cdot(2-2g).$$
But $\sh{U}$ is the composition of the blow-up of $\pp^2$ at $r$ points together with the blow-up at $|Z|$ points on $S$. Thus, 
$$\chi(\sh{U},\zz)=3+r+|Z|\geq 3,$$
which is a contradiction. It follows that $\pi:\sh{U}\to \pp^1$ must have singular fibers. This proves non-isotriviality. Fibers of $\pi$ might still be curves isomorphic to $C$ with a rational tail attached, in which case the moduli map induced by the pencil is still trivial. This cannot happen: indeed, if we see the pencil as a $\pp^1$-family of $r$-nodal sextics in $\pp^2$, if $C\sim R+C'$ where $R$ is an irreducible rational tail, then $R$ is a line or a conic on $\pp^2$ in which case the residual curve $C'$ drops in genus. Alternatively, the same argument as in the proof of Lemma \ref{nontrivial} applies. 
\end{proof}


\section{Unirationality in Small Genus}

Recall that a variety is unirational if it is dominated by a rational variety. As before, $\mu=(m_1,\ldots,m_n)$ is an holomorphic partition of $2g-2$, with $g\geq3$. We will assume that the length of the partition is at most $g-1$ and that $\Hgmu$ is connected. A similar argument as in \cite[Thm 3.1]{2010arXiv1004.0278F} can be used to obtain unirationality for $3\leq g\leq6$. The strategy is to construct a projective bundle $\cal{P}_g$ over a rational variety $\Sigma$ that dominates $\Hgmubar$.\\

For $3\leq g \leq 6$, we have $\rho(g,2,6)\geq 0$. By choosing $[C,y_1,\ldots,y_{g-1}]\in \mathcal{M}_{g,g-1}$ and $A\in \mathrm{G}_6^2(C)$ general, we can assume that the map $\phi_A:C\to \Gamma\subset \pp^2$ realizes $C$ as a $(10-g)$-nodal sextic and the marked points $y_1,\ldots,y_{g-1}$ are disjoint from the preimages of the nodes $\phi_A^{-1}\left(\mathrm{Sing}(\Gamma)\right)$. Consider the following diagram of rational maps

\begin{displaymath}
\xymatrix{
&\cal{P}_g\ar@{-->}[dr]^{\pi_2}\ar@{-->}[dl]_{\nu_g}&\\
\cal{M}_{g,n}&&\Sigma,
}
\end{displaymath}
where $\Sigma\subset |\mathcal{O}_{\pp^2}(3)|\times (\pp^2)^9$ is defined as
$$\Sigma:=\left\{(E,x_1,\ldots,x_\delta,y_1,\ldots,y_{g-1})\in |\mathcal{O}_{\pp^2}(3)|\times (\pp^2)^9\mid x_1,\ldots,y_{g-1}\in E \right\}$$
and the fiber of $\pi_2$ over a general point $(E,\bar{x},\bar{y})\in\Sigma$ is the linear space of plane sextics $\Gamma$, nodal at $x_1,\ldots,x_\delta$ and with $\mu$ contact at the points $y_1,\ldots,y_{g-1}$ with the cubic $E$. In other words,  
$$\pi_2^{-1}(E,\bar{x},\bar{y}):=\left\{\Gamma\in |\mathcal{O}_{\pp^2}(6)|\left|\hspace{0,1cm}\Gamma\hbox{ is nodal at }\bar{x}\hbox{ and }\Gamma\cdot E=\sum_{1}^{g-1}m_iy_i+2(x_1+\ldots+x_{\delta})\right.\right\}.$$
If the partition $\mu$ has length less than $g-1$, we complete it with zeros so that the length is $g-1$. \\

One can see that the map induced by the projection $\Sigma\to(\pp^2)^{9}$ is birational. The map $\nu_g$ sends $(\Gamma,E,\bar{x},\bar{y})$ to $[C,\tilde{y}]$, where $C$ is the normalization of $\Gamma$ and $\tilde{y}$ is obtained by omitting from $(y_1,\ldots,y_{g-1})$ the terms $y_i$ where $m_i=0$.

\begin{proposition}
\label{prop3.1}
The map $\pi_2$ is dominant.
\end{proposition}
\begin{proof}
Let $(E,\bar{x},\bar{y})$ be a general point in $\Sigma$ and $\eps:X\to \pp^2$ be the blow-up of $\pp^2$ at $x_1+\ldots+x_{\delta}$ with $F$ the exceptional divisor and $L$ the pull back of the line. Let $\mathcal{K}(\mu)$ be the kernel of the composition 
$$\mathcal{O}_X(6L-2F)\to\mathcal{O}_E(6L-2F)\to\mathcal{O}_E(6L-2F)\mid_{\sum m_iy_i}.$$
Since $E=3L-F=-K_X$ and $4\leq \delta\leq 7$, we obtain $h^1(-K_X)=0$. From this one deduces that $H^0$ of the first map is surjective. If the general element of $\pp(H^0(X,\mathcal{K}(\mu)))$ is nodal, then 
$$\dim \pi_2^{-1}(E,\bar{x},\bar{y})= h^0(\mathcal{K}(\mu))-1\geq 27-3\delta-2(g-1)=9-\delta.$$
By specialization, it is enough to show that the general element of the linear system $|\mathcal{K}(\mu)|$ is nodal when the points coincide, that is, $y_1=\ldots=y_{g-1}=y$. In this case, the fiber of $\pi_{2}$ consists of $\delta$-nodal sextics, with nodes at $\bar{x}$ and intersecting the tangent line $\ell_{y}$ of $E$ at $y$ with order $2g-2$. This is a \textit{generalized Severi variety}, and by \cite[Prop. 2.1]{CaHa98}, the space of $\delta$-nodal sextics with such prescribed tangency with a given line in non-empty, in particular $|\mathcal{K}(2g-2)|$ contains nodal curves and so does $|\mathcal{K}(\mu)|$. 

\end{proof}

Now we can give a proof of unirationality in the established range.

\begin{proof}[Proof of the unirationality column in Theorem \ref{thm}]
Let $3\leq g\leq 6$ and $\delta=10-g$. We need to prove that $\nu_g$ dominates $\Hgmu$. Let $(\Gamma,E,\bar{x},\bar{y})\in \mathcal{P}_g$, with $\Gamma\cdot E=\sum m_iy_i+2(x_1+\ldots+x_\delta)$ and $\nu:C\to \Gamma$ the normalization. Then
$$\cal{O}_C\left(\sum m_iy_i\right)\cong\nu^*\cal{O}_{\Gamma}(3)\left(-\nu^{-1}(\bar{x})\right)$$
and by adjunction 
$$\sum m_iy_i\sim K_C.$$
Thus, the image of $\nu_g$ lies in $\Hgmu$. Now we prove dominance. Let $[C,\tilde{y}]$ be a general point in $\Hgmu$. By assumption on $g$, we can choose $A\in G^2_6(C)$ general so that the associated map $\phi_A:C\to\pp^2$ realizes $C$ as a $\delta$-nodal curve and by adjunction
$$h^0\left(C,\nu^*\cal{O}_{\Gamma}(3)\left(-\nu^{-1}(\bar{x})-\sum_1^{n} m_iy_i\right)\right)=1.$$
Thus, there exists a cubic $E\in|\mathcal{O}_{\pp^2}(3)|$ such that 
$$\Gamma\cdot E=\sum m_iy_i+2(x_1+\ldots+x_{\delta}).$$
By adding with multiplicity zero any $g-1-n$ points on $E$, we have that the image of 
$$(\phi_A(C), x_1,\ldots,x_{\delta},\bar{y})\in \mathcal{P}_g$$ 
is exactly $[C,\tilde{y}]\in \Hgmu$.
\end{proof}

\section{Non-irreducible Strata}

Let $\mu=(2,\ldots,2)$ be a partition of $2g-2$. There is a natural map from $\Hgmu$ to the space of spin curves, which splits into two connected components depending on the parity of the theta characteristic,
$$
\begin{array}{rcl}
\pi:\Hgmu&\to&\sh{S}_g^+\coprod\sh{S}_g^{-}\\
\left[C,x_1,\ldots,x_{g-1}\right]&\mapsto&\left[C,\cal{O}_{C}\left(x_1+\ldots+x_{g-1}\right)\right].
\end{array}
$$

A general point $[C,\nu]\in\sh{S}_g^{+}$ satisfies $h^0(\nu)=0$ and the condition 
$h^0(\nu)\geq 2$ is divisorial. The divisor $x_1+\ldots+x_{g-1}\in\hbox{Div}(C)$ is effective and for a general point on the image of $\cal{H}^+\to\sh{S}_g^+$, we have that $h^{0}(C,x_1+\ldots+x_{g-1})=2$. See \cite[Thm. 0.2]{2008arXiv0805.2424F}. Thus, $\overline{\cal{H}}^+/S_{g-1}$ is a $\pp^1$-bundle over an open subset of a divisor $Z$ in $\sh{S}_g^+$. In particular, for every genus, through a general point of $\overline{\cal{H}}^+/S_{g-1}$ passes a rational curve.

\subsection{Hyperelliptic Strata.}

Let $\sh{H}_{g,1}\subset \cal{M}_{g,1}$ be the space of pairs $[C,p]$, where $C$ is an hyperelliptic curve and $p\in C$ is a Weierstrass point on it. The locus $\sh{H}_{g,1}$ can be defined as the fiber product over $\Mg$ of the Weierstrass divisor 
$$\sh{W}:=\left\{[C,p]\in\mathcal{M}_{g,1}\mid p\hbox{ is a Weierstrass point on $C$} \right\}\subset \mathcal{M}_{g,1}$$ 
and the hyperelliptic locus $\hbox{Hyp}_{g,1}\subset\cal{M}_{g,1}$
$$\sh{H}_{g,1}=\hbox{Hyp}_{g,1}\times_{\Mg}\sh{W}.$$

Recall that the space $\hbox{Hyp}_g\subset\Mg$ is birational to $\cal{M}_{0,2g+2}\big/S_{2g+2}$. In particular there is a dominant map 
$$\cal{M}_{0,2g+2}\to\sh{H}_{g,1}$$
sending the tuple of points $(p_1,\ldots,p_{2g+2})$ on $\pp^1$ to the unique double cover ramified over $p_1+\ldots+p_{2g+2}$ 
$$f:C\to\pp^1$$
together with the Weierstrass point $f^{-1}(p_1)$. This gives us unirationality of $\sh{H}_{g,1}$. On the other hand, it is easy to see that $p$ is a Weierstrass point on a hyperelliptic curve $C$ if and only if $(2g-2)p\sim K_C$, i.e., 
$$\sh{H}_{g,1}=\cal{H}_g^{hyp}(2g-2).$$

\end{document}